\documentclass[11pt]{article}
\usepackage[utf8]{inputenc}
\usepackage{amsmath,amsfonts}
\usepackage{amsthm}
\usepackage{amssymb}
\usepackage{graphicx}
\usepackage{enumitem}
\usepackage{genyoungtabtikz}
\usepackage{ytableau}
\usepackage{indentfirst}
\usepackage{xy}
\usepackage{xypic}
\usepackage{float}
\allowdisplaybreaks

\def\Allowcells、
\def\borderhooks{boundary hooks}

\newtheorem{thm}{Theorem}
\newtheorem{ex}[thm]{Example}

\newtheorem{lem}[thm]{Lemma}
\newtheorem{cor}[thm]{Corollary}

\newtheorem{defn}[thm]{Definition}

\newtheorem{theorem}[thm]{Theorem}

\numberwithin{thm}{section}

\definecolor{lightblue}{rgb}{.85,.9,.96}
\definecolor{lightpurple}{rgb}{.9,.5,.9}
\definecolor{lightred}{rgb}{.9,.7,.7}
\definecolor{grey}{rgb}{.85,.85,.85}

\Yboxdimx{30pt}
\Ylinecolour{black}
\Ynodecolour{black!50!black}

\usepackage[colorinlistoftodos]{todonotes}

\DeclareMathOperator{\I}{\mathfrak{S}}

\DeclareMathOperator{\Sym}{Sym}

\DeclareMathOperator{\NSym}{NSym}

\DeclareMathOperator{\QSym}{QSym}

\DeclareMathOperator{\ndet}{\mathfrak{det}}

\DeclareMathSymbol{\shortminus}{\mathbin}{AMSa}{"39}

\author{Sarah Mason \and Jack Xie}
\title{A classification of nonzero skew immaculate functions}

\begin{document}
\maketitle

\paragraph{Abstract}

This article presents conditions under which the skewed version of immaculate noncommutative symmetric functions are nonzero.  The work is motivated by the quest to determine when the matrix definition of a skew immaculate function aligns with the Hopf algberaic definition.  We describe a necessary condition for a skew immaculate function to include a non-zero term, as well as a sufficient condition for there to be at least one non-zero term that survives any cancellation.  We bring in several classical theorems such as the Pigeonhole Principle from combinatorics and Hall's Matching Theorem from graph theory to prove our theorems.

\section{Background and Introduction}

A function with $n$ commuting variables is said to be \emph{symmetric} if it remains the same when the variables are permuted.  That is, the function $f(x_1, x_2, \hdots , x_n)$ is symmetric if $f(x_{\sigma_1}, x_{\sigma_2}, \hdots , x_{\sigma_n}) = f(x_1, x_2, \hdots , x_n)$ for any permutation $\sigma=(\sigma_1, \sigma_2, \hdots , \sigma_n)$ of the variable indices.   See Sagan~\cite{Sag01} or Stanley~\cite{Sta99} for an excellent introduction to symmetric functions.

One example of a symmetric function on three variables is $$f(x_1,x_2,x_3)=x_1^3x_2^2+x_1^3x_3^2+x_2^3x_3^2+x_1^2x_2^3+x_1^2x_3^3+x_2^2x_3^3,$$ as when we apply the permutation $(2,1)$ to the subscripts of $f(x_1,x_2,x_3),$ we get $$f(x_2,x_1,x_3)= x_2^3x_1^2+x_2^3x_3^2+x_1^3x_3^2+x_2^2x_1^3+x_2^2x_3^3+x_1^2x_3^3,$$ which is the same as $f(x_1,x_2,x_3)$. Notice the function $g(x_1,x_2)=x_1+x_2^2$ is not a symmetric function on two variables since applying the permutation $(2,1)$ produces $$g(x_2,x_1) = x_2+x_1^2 \not= g(x_1,x_2).$$

The ring (or vector space) of all symmetric functions is called $\Sym$.  Bases for $\Sym$ are indexed by \emph{partitions}, where a partition of $k$ is a weakly decreasing sequence of positive integers that sum to $k$.  The elements in such a sequence are referred to as its \emph{parts}, and the \emph{length} of a partition is the number of parts.  For example, $\lambda=(5,3,3,2)$ is a partition of $13$ with length $4$.  We write $\lambda \vdash n$ to indicate that $\lambda$ is a partition of $n$.  

In this article, we typically work over finitely many variables (often $x_1, x_2, \hdots , x_n)$ so that is the assumption if the variables are not given.  However, it is straightforward to extend these definitions to infinitely many variables.

The \emph{monomial symmetric functions} are one of the most natural bases for Sym.  The monomial symmetric function on $n$ variables indexed by a partition $\lambda$ is defined by $$m_{\lambda} (x_1, x_2, \hdots , x_n) = \sum x_{i_1}^{\lambda_1} x_{i_2}^{\lambda_2} \cdots x_{i_{\ell}}^{\lambda_{\ell}},$$ where the sum is over all distinct monomials with exponents $\lambda_1, \lambda_2, \hdots , \lambda_{\ell}$.  For example, $m_{211}(x_1,x_2,x_3)={x_1}^2{x_2}{x_3}+{x_1}{x_2}^2{x_3}+{x_1}{x_2}{x_3}^2$.

Another important basis for the vector space $\Sym$ is the \emph{complete homogeneous symmetric functions}, $\{ h_{\lambda} \}_{\lambda \vdash n}$.  The complete homogeneous basis is defined by $$h_k(x_1, x_2, \hdots , x_n)=\sum_{1 \le i_1 \le i_2 \le \cdots \le i_k \le n} x_{i_1} x_{i_2} \cdots x_{i_k}, \; \; \; \textrm{with}$$  $$h_{\lambda}(x_1, \hdots , x_n) = h_{\lambda_1} h_{\lambda_2} \cdots h_{\lambda_{\ell}}.$$ 

For example, $h_{3,1}(x_1,x_2,x_3)=h_3(x_1,x_2,x_3)h_1(x_1,x_2,x_3)=(x_1^3+x_2^3+x_3^3+x_1^2x_2 + x_1x_2^2+x_1^2x_3+x_1x_3^2+x_2^2x_3+x_2x_3^2+x_1x_2x_3)(x_1+x_2+x_3).$

The \emph{Schur functions} are another important basis for $\Sym$ with connections to many different branches of mathematics and beyond, including representation theory and algebraic geometry.  One way to define the Schur functions is through diagrams called \emph{semi-standard Young tableaux}.

To define a semi-standard Young tableau, we first describe the \emph{Ferrers diagram} of a partition.  Given a partition $\lambda=(\lambda_1, \hdots , \lambda_{\ell})$, draw a collection of left-justified boxes (often called \emph{cells}) so that there are $\lambda_i$ boxes in row $i$.  (Here, we let $\lambda_1$ be the bottom row and read from bottom to top, consistent with the \emph{French notation}.)  Such a diagram is called the \emph{Ferrers diagram} for the partition.  For example, the Ferrers diagram of $\lambda=(5,3,2,2)$ is depicted in Figure~\ref{fig:Ferrers}.

\begin{figure}
\begin{picture}(200,50)
\put(10,0){
\hbox{
\vbox{
\Yfillcolour{white}
%\young({7,1}{7,2}{7,3}{7,4}{7,5}{7,6},)
\young(\ \ ,)
\vskip -2.5pt
\Yfillcolour{white}
%\young({6,1},)
\young(\ \ ,)
\vskip -2.5pt
\Yfillcolour{white}
\young(\ \ \,)
\vskip -2.5pt
\Yfillcolour{white}
\young(\ \ \ \ \,)
}}}

\put(190,0){
\hbox{
\vbox{
\Yfillcolour{white}
%\young({7,1}{7,2}{7,3}{7,4}{7,5}{7,6},)
\young({7}{9},)
\vskip -2.5pt
\Yfillcolour{white}
%\young({6,1},)
\young({6}{7},)
\vskip -2.5pt
\Yfillcolour{white}
\young({3}{5}{5},)
\vskip -2.5pt
\Yfillcolour{white}
\young({2}{2}{2}{3}{5},)
}}}
\end{picture}
\label{fig:Ferrers}
\caption{The figure on the left is the Ferrers diagram for $\lambda=(5,3,2,2)$.  The figure on the right is a semi-standard Young tableau of shape $\lambda=(5,3,2,2)$ and weight $x_2^3x_3^2x_5^3x_6x_7^2x_9$.}
\end{figure}
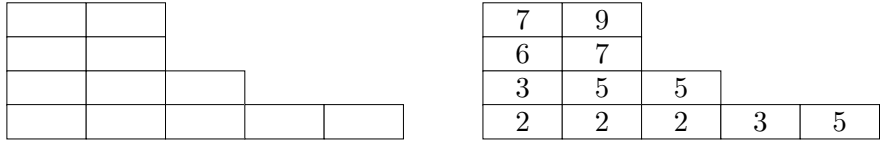

A \emph{semi-standard Young tableau} of shape $\lambda$ is a filling of the cells of the Ferrers diagram of $\lambda$ with positive integers so that the following conditions hold.  
\begin{enumerate}
\item The entries in each row weakly increase from left to right.
\item The entries in each column strictly increase from bottom to top.
\end{enumerate}

Each semi-standard Young tableau $\tau$ is assigned a weight $x^{\tau}=x_1^{a_1} x_2^{a_2} \cdots x_n^{a_n},$ where $a_i$ is the number of times $i$ appears in $\tau$.  See Figure~\ref{fig:Ferrers} for an example of a semi-standard Young tableau and its weight.

The \emph{Schur function} $s_{\lambda}$ is defined according to the following generating function. $$s_\lambda(x)=\sum_{\tau\in SSYT(\lambda)} x^{\tau},$$ where $SSYT(\lambda)$ is the set of all possible semi-standard young tableaux of shape $\lambda$.  This means that the Schur function indexed by shape $\lambda$ is the sum of all the monomials corresponding to the semi-standard young tableaux of shape $\lambda$. For example, the following semi-standard Young tableaux of shape $(2,1)$ on three variables produce the Schur function $s_{21}=x_1^2x_2 + x_1^2x_3+x_1x_2^2+2x_1x_2x_3 + x_1x_3^2+x_2^2x_3+x_2x_3^2 =m_{21}+2m_{111}$. 

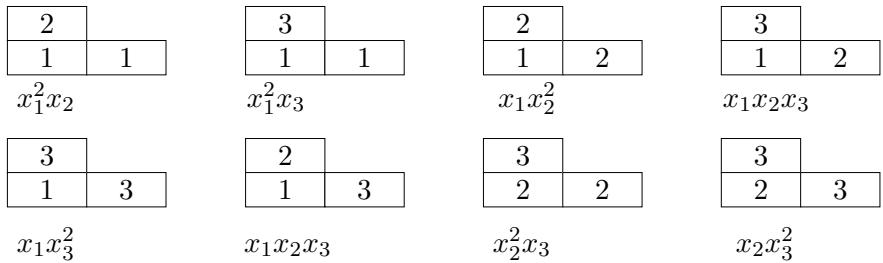
\begin{picture}(400,120)
\put(-20,30){
\hbox{
\vbox{
\Yfillcolour{white}
\young({3} \  )
\vskip -2.5pt
\Yfillcolour{white}
\young({1}{3} \  )
}}}
\put(70,30){
\hbox{
\vbox{
\Yfillcolour{white}
\young({2} \  )
\vskip -2.5pt
\Yfillcolour{white}
\young({1}{3} \  )
}}}
\put(160,30){
\hbox{
\vbox{
\Yfillcolour{white}
\young({3} \  )
\vskip -2.5pt
\Yfillcolour{white}
\young({2}{2} \  )
}}}
\put(250,30){
\hbox{
\vbox{
\Yfillcolour{white}
\young({3} \  )
\vskip -2.5pt
\Yfillcolour{white}
\young({2}{3} \  )
}}}

\put(8,10){$x_1x_3^2$}
\put(94,10){$x_1x_2x_3$}
\put(188,10){$x_2^2x_3$}
\put(280,10){$x_2x_3^2$}
\put(8,65){$x_1^2x_2$}
\put(95,65){$x_1^2x_3$}
\put(190,65){$x_1x_2^2$}
\put(275,65){$x_1x_2x_3$}

\put(-20,80){
\hbox{
\vbox{
\Yfillcolour{white}
\young({2} \  )
\vskip -2.5pt
\Yfillcolour{white}
\young({1}{1} \  )
}}}
\put(70,80){
\hbox{
\vbox{
\Yfillcolour{white}
\young({3} \  )
\vskip -2.5pt
\Yfillcolour{white}
\young({1}{1} \  )
}}}
\put(160,80){
\hbox{
\vbox{
\Yfillcolour{white}
\young({2} \  )
\vskip -2.5pt
\Yfillcolour{white}
\young({1}{2} \  )
}}}
\put(250,80){
\hbox{
\vbox{
\Yfillcolour{white}
\young({3} \  )
\vskip -2.5pt
\Yfillcolour{white}
\young({1}{2} \  )
}}}
\end{picture}

We can also define Schur functions indexed by \emph{skew shapes}.  Given partitions $\lambda$ and $\nu$, if $\nu_i \le \lambda_i$ for all $i$, then $\lambda/\nu$ is the shape obtained by first drawing the diagram for $\lambda$ and then shading out the cells that are also contained in the diagram of $\nu$. 
 For example, the diagrams in Figure~\ref{fig:skewschur} are all semi-standard Young tableau of shape $22/1$.  Compute the skew Schur function indexed by $\lambda/\nu$ by filling the non-shaded cells with positive integers so that again the row entries weakly increase left to right and the column entries strictly increase bottom to top. 

\begin{figure}
\begin{picture}(400,80)
\put(0,20){
\hbox{
\vbox{
\Yfillcolour{white}
\young({3}{3} \  )
\vskip -2.5pt
\Yfillcolour{white}
\young(!<\Yfillcolour{grey}>\  !<\Yfillcolour{white}>{1} \  )
}}}
\put(90,20){
\hbox{
\vbox{
\Yfillcolour{white}
\young({1}{3} \  )
\vskip -2.5pt
\Yfillcolour{white}
\young(!<\Yfillcolour{grey}>\  !<\Yfillcolour{white}>{2} \  )
}}}
\put(180,20){
\hbox{
\vbox{
\Yfillcolour{white}
\young({2}{3} \  )
\vskip -2.5pt
\Yfillcolour{white}
\young(!<\Yfillcolour{grey}>\  !<\Yfillcolour{white}>{2} \  )
}}}
\put(270,20){
\hbox{
\vbox{
\Yfillcolour{white}
\young({3}{3} \  )
\vskip -2.5pt
\Yfillcolour{white}
\young(!<\Yfillcolour{grey}>\  !<\Yfillcolour{white}>{2} \  )
}}}

\put(28,0){$x_1x_3^2$}
\put(114,0){$x_1x_2x_3$}
\put(208,0){$x_2^2x_3$}
\put(300,0){$x_2x_3^2$}
\put(28,55){$x_1^2x_2$}
\put(115,55){$x_1^2x_3$}
\put(210,55){$x_1x_2^2$}
\put(295,55){$x_1x_2x_3$}

\put(0,70){
\hbox{
\vbox{
\Yfillcolour{white}
\young({1}{2} \  )
\vskip -2.5pt
\Yfillcolour{white}
\young(!<\Yfillcolour{grey}>\  !<\Yfillcolour{white}>{1} \  )
}}}
\put(90,70){
\hbox{
\vbox{
\Yfillcolour{white}
\young({1}{3} \  )
\vskip -2.5pt
\Yfillcolour{white}
\young(!<\Yfillcolour{grey}>\  !<\Yfillcolour{white}>{1} \  )
}}}
\put(180,70){
\hbox{
\vbox{
\Yfillcolour{white}
\young({2}{2} \  )
\vskip -2.5pt
\Yfillcolour{white}
\young(!<\Yfillcolour{grey}>\  !<\Yfillcolour{white}>{1} \  )
}}}
\put(270,70){
\hbox{
\vbox{
\Yfillcolour{white}
\young({2}{3} \  )
\vskip -2.5pt
\Yfillcolour{white}
\young(!<\Yfillcolour{grey}>\  !<\Yfillcolour{white}>{1} \  )
}}}
\end{picture}
\caption{The skew Schur function $s_{22/1}(x_1,x_2,x_3)=x_1^2x_2 + x_1^2x_3+x_1x_2^2+x_1x_2x_3 + x_1x_3^2+x_1x_2x_3+x_2^2x_3+x_2x_3^2$ is generated by the diagrams above.}
\label{fig:skewschur}
\end{figure}
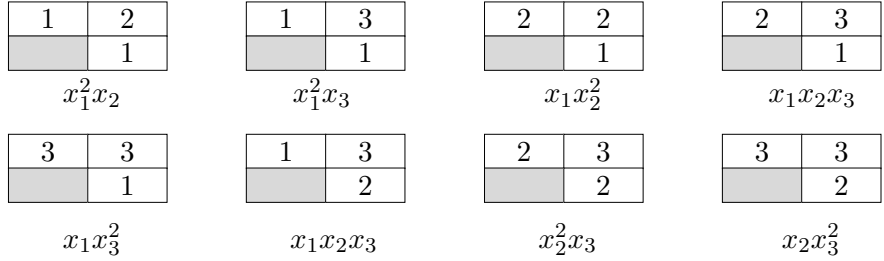

The \emph{Jacobi-Trudi Formula} provides the expansion of a Schur function into complete homogeneous functions.  In the following, to reduce to the situation of a Schur function indexed by a single partition (rather than a skew shape), set $\nu=(0,0, \hdots , 0)$.

\begin{theorem}[Jacobi-Trudi Formula]{\cite{Mac95}}
Let $\lambda/\nu$ be a skew shape.  Then 
\begin{equation}
s_{\lambda/\nu}=\det(h_{\lambda_i-i-(\nu_j-j)})_{i,j}.
\label{E:SymJacobiTrudi}\end{equation}\label{T:SymJacobiTrudi}
\end{theorem}

For example, let $\lambda=(2,2)$ and $\nu=(1)$.  Then we obtain the matrix $\begin{bmatrix} 
h_1&h_3 \\
h_0&h_2
\end{bmatrix}$.  Taking the determinant of this matrix produces the Schur function expansion $s_{22/1}=h_1h_2-h_3h_0=h_{2,1}-h_3$.

Since Schur functions form a basis for symmetric functions, the product of two arbitrary Schur functions can be written as a sum of Schur functions.  The coefficients appearing in this expansion are called \emph{Littlewood-Richardson coefficients}.  In particular, if $$s_{\nu} s_{\mu} = \sum_{\lambda} c_{\nu, \mu}^{\lambda} s_{\lambda},$$ then the terms $c_{\nu, \mu}^{\lambda}$ are the Littlewood-Richardson coefficients.  These numbers are also connected to skew Schur functions, as $$s_{\lambda/\nu} = \sum_{\mu} c_{\nu, \mu}^{\lambda} s_{\mu}.$$

The following example demonstrates this relationship.  Consider the Schur function product 
$$s_{32}s_{51}=s_{5321}+s_{533}+s_{542}+s_{6221}+s_{6311}+2s_{632}+s_{641}+s_{722}+s_{7211}+$$ $$2s_{731}+s_{74}+s_{821}+s_{83}$$ and skew Schur function expansion $$s_{632/51}=s_{221}+s_{311}+2s_{32}+s_{41}.$$ 
 Notice the coefficient of $s_{632}$ in the product $s_{32}s_{51}$ is exactly the same as the coefficient of $s_{32}$ in the expansion of the skew $s_{632/51}$.

In what follows, we will need several generalizations of partitions, which can be thought of as special cases of integer sequences.  A partition is a weakly decreasing sequence of positive integers, while a \emph{composition} is an ordered sequence of positive integers; the entries in a composition are not required to be weakly decreasing.  A \emph{weak composition} is an ordered sequence of non-negative integers.

The ring of symmetric functions can be generalized in a number of ways.  One generalization is called the ring $\NSym$ of \emph{noncommutative symmetric functions}, which can be thought of as a ring or vector space over real numbers generated by the \emph{complete homogeneous noncommutative symmetric} functions $H_0, H_1, H_2, \hdots$ with no relations.  That is, $\NSym=\mathbb{R} \langle H_0, H_1, \hdots \rangle$ where $H_a=0$ if $a<0$ and $H_0=1$. In this paper, we use the terminology ``zero terms", ``terms equalling zero" and ``$H$ negatives" interchangeably, but they all mean H functions with negative subscripts, since $H_a=0$ for all $a<0$.  Since there are no relations on the $H$ functions, we have $H_aH_b \not= H_b H_a$, and we define $H_{(a_1,a_2, \hdots , a_{\ell})}=H_{a_1}H_{a_2} \cdots H_{a_{\ell}}$.  The elements $H_{\alpha}$ (indexed by compositions) form a basis for $\NSym$.

Another important basis for $\NSym$ is the set of \emph{immaculate functions}.  Immaculate functions were introduced by Berg, Bergeron, Saliola, Serrano, and Zabrocki through creation operators~\cite{BBSSZ14} but can be described in terms of the homogeneous functions via an analogue of the Jacobi-Trudi Formula.  We take the following formula as the definition.

\begin{defn}{\cite{BBSSZ14}}
Let $\mu \in \mathbb{Z}^k$ be a sequence of integers and let $(M_{\mu})_{i,j}=H_{(\mu_i-i+j)}$.  Then the \emph{immaculate function}  $\I_{\mu}$ is given by
\begin{equation}
\I_{\mu}=\ndet(M_{\mu}),
\end{equation}
where the noncommutative determinant $\ndet$ is computed using Laplace expansion starting in the top row and continuing sequentially to the bottom row.
\end{defn}

The immaculate functions reduce to the Schur functions under the \emph{forgetful map} from $\NSym$ to $\Sym$, which ``forgets" that the variables don't commute.  The vector space dual to $\NSym$ is a called the \emph{quasisymmetric functions}, often abbreviated by $\QSym$.  A function $f$ is said to be \emph{quasisymmetric} if for all compositions $\alpha=(\alpha_1, \alpha_2, \hdots , \alpha_{\ell})$, the coefficient of $x_{1}^{\alpha_1} x_{2}^{\alpha_2} \cdots x_{\ell}^{\alpha_{\ell}}$ in $f$ is equal to coefficient of $x_{j_1}^{\alpha_1} x_{j_2}^{\alpha_2} \cdots x_{j_{\ell}}^{\alpha_{\ell}}$ in $f$ for any sequence of positive integers $1 \le j_1 < j_2 < \cdots < j_{\ell}$.

A polynomial in $n$ variables is said to be \emph{quasisymmetric} if the coefficient of the monomial $x_1^{\alpha_1} x_2^{\alpha_2} \cdots x_{\ell}^{\alpha_{\ell}}$ is equal to the coefficient of $x_{i_1}^{\alpha_1} x_{i_2}^{\alpha_2} \cdots x_{i_{\ell}}^{\alpha_{\ell}}$ for all $i_1 < i_2< \cdots < i_{\ell}$ and all compositions $\alpha$.  The \emph{dual immaculate basis} $\{\I_{\alpha}^{\star}\}_{\alpha}$ in $\QSym$ is the basis dual to the immaculate basis. 
 One open problem is to find a formula for the coefficients appearing when the product of two dual immaculates is expanded into the dual immaculate basis.  Allen and Mason generalized the Jacobi-Trudi formula for skew Schur functions to $\NSym$ to produce the following candidate for skew immaculate functions.

\begin{defn}{\label{def:skewdef}}{\cite{AllMas23}}
Let $\alpha$ and $\beta$ be compositions and set $(M_{\alpha/\beta})_{i,j}=H_{(\alpha_i-i)-(\beta_j-j)}$.  We call this the \emph{matrix associated to $\alpha$ and $\beta$} (or sometimes just the \emph{associated matrix} for brevity).  Then the \emph{skew immaculate function} is given by
\begin{equation}
\I_{\alpha/\beta}=\ndet(M_{\alpha/\beta}),
\end{equation}
where the noncommutative determinant $\ndet$ is expanded using Laplace expansion starting in the top row and continuing sequentially to the bottom row.
\end{defn}

For example, let $\alpha=(6,4,3)$ and let $\beta=(2,4,1)$.  Then
$\I_{643/241}=\ndet(M_{643/241})$.  The associated matrix is given by
$$M_{643/241}=\begin{bmatrix} H_4 & H_3 & H_7 \\ H_1 & H_{0} & H_4 \\ H_{-1} & H_{-2} &H_{2}\end{bmatrix}.$$

By Laplace expansion, $\ndet(M_{643/241})=H_4 H_0 H_2-H_4 H_4 H_{-2}-H_3 H_1 H_2+H_3 H_4H_{-1}+H_7 H_1 H_{-2}-H_7 H_0 H_{-1}$.  Since $H_k=0$ for all $k<0$ and $H_a H_b=H_{ab}$, we have $\I_{643/241}=\ndet(M_{643/241})=H_{42}-H_{312}$.

Unfortunately, the coefficients appearing when a skew immaculate is written as a sum of immaculates do not always align with those appearing in products of dual immaculates.  For example, $\I_{33/22}= \I_{11}$, however, the term $\I_{33}$ appears in the product $\I_{22}^{\star} \I_{2}^{\star}$ as well as the product $\I_{22}^{\star} \I_{11}^{\star}$.  If the skew immaculate construction were compatible with multiplication in the dual, we would expect the coefficient of $\I_{2}$ in $\I_{33/22}$ to also equal $1$.  It is an open problem to classify which skew immaculates correspond to multiplication in this way.  Skew immaculate functions equal to $0$ generally do not correspond to multiplication, so it is useful to focus on non-zero skew immaculate functions.  There are a large number of skew immaculates equal to zero, to the extent that including them in computations significantly slows down processing speed.  Therefore, it would be useful to have a formula to rule out classes of skew shapes $\alpha/\beta$ for which $\I_{\alpha/\beta}=0$.  The main goal of this paper is therefore to describe a collection of pairs $(\alpha,\beta)$ of compositions such that $\I_{\alpha/\beta} \not= 0$.

In Section~\ref{sec:necessary}, we provide a necessary condition for the $H$-basis expansion of $\I_{\alpha/\beta}$ to include at least one non-zero term.  Section~\ref{sec:Hall} provides background on Hall's Matching Theorem.  In Section~\ref{sec:sufficient}, we use Hall's Matching Theorem to prove that in fact our conditions for a nonzero term are sufficient, but it is possible for cancellation to occur.  Finally in Section~\ref{sec:cancel} we describe a class of composition pairs for which the corresponding immaculate function is nonzero even after cancellation.

\section{Determinants in $\NSym$}{\label{sec:necessary}}

Our starting point is the expansion of a skew immaculate into the complete homogeneous basis for $\NSym$, given by Definition~\ref{def:skewdef}.  A skew immaculate function $\I_{\alpha/\beta}$ will be zero if either $H_{\gamma}=0$ for each $H_{\gamma}$ appearing in the $H$-basis expansion of $\I_{\alpha/\beta}$ or if all of the non-zero terms cancel out.

In the following, we abuse notation by referring to the complete homogeneous noncommutative symmetric functions as simply the homogeneous basis for $\NSym$ or even the $H$-basis.  We first describe a sufficient condition for every $H_{\gamma}$ to be zero.  To do this, we need the following definition.

\begin{defn}
Let $\alpha=(\alpha_1, \alpha_2, \hdots, \alpha_{\ell})$ be a composition.  Define $\hat{\alpha}=(\alpha_1-1, \alpha_2,-2 , \hdots , \alpha_{\ell}-\ell)$ to be the sequence obtained by subtracting $i$ from the $i^{th}$ part of $\alpha$ for all $i$.
\end{defn}

\begin{ex}{\label{ex:negs}}
If $\alpha=(3,2,3,5,1)$ then $\hat{\alpha}=(2,0,0,1,-4))$.  
\end{ex}
Notice if every term in the expansion of a skew immaculate function into the homogeneous basis is $0$, then there is a relationship between the number of negative entries in the expansion matrix and the number of rows of the matrix that contain the negative entries. For example, if $\alpha=(5,7,1,3)$ and $\beta=(5,5,5,1)$, we get $\hat{\alpha}=(4,5,-2,-1)$ and $ \hat{\beta}=(4,3,2,-3)$. This results in the following matrix, whose noncommutative determinant equals zero.

$$\begin{bmatrix} H_0 & H_1 & H_2 & H_7 \\ H_1 & H_{2} & H_3 & H_8 \\ H_{-6} & H_{-5} &H_{-4} & H_1 \\ H_{-5}&H_{-4}&H_{-3}&H_2\end{bmatrix} \hspace*{0.5in}$$

Notice that two of the rows in this matrix contain negative entries and in each such row there are three negative entries.  It is therefore natural to look for a general condition on the number of negative entries in the rows containing negative entries which forces all the terms in the $H$-basis expansion to be zero.  The following theorem makes this idea precise.

\begin{theorem}{\label{thm:zeros}}
Let $\alpha=(\alpha_1,...,\alpha_{\ell})$ and $\beta=(\beta_1,...,\beta_{\ell})$ be compositions.  If there is at least one non-zero term (before cancellation) in the $H$-basis expansion of $\I_{\alpha/\beta}$, then for every $k$ between $1$ and $\ell$ inclusive, set of all parts of $\hat{\alpha}$ which are smaller than at least $\ell-k+1$ entries in $\hat{\beta}$ has cardinality less than or equal to $k-1$.
\end{theorem}

Theorem~\ref{thm:zeros} provides a necessary condition for a skew immaculate to be non-zero. However, before we prove Theorem~\ref{thm:zeros}, it is helpful to look at a simpler case to better understand the arguments the main proof.

\begin{lem}{\label{lem:onebigger}}
Let $\alpha=(\alpha_1,...,\alpha_{\ell})$ and $\beta=(\beta_1,...,\beta_{\ell})$ be compositions.  If there is at least one non-zero term (before cancellation) in the $H$-basis expansion of $\I_{\alpha/\beta}$, then for every entry $a \in \hat{\alpha}$, there exists an entry $b$ in $\hat{\beta}$ such that $a \ge b$.
\end{lem}

\begin{proof}
Let $\alpha=(\alpha_1,...,\alpha_l)$ and $\beta=(\beta_1,...,\beta_l)$ be compositions.  To prove Lemma~\ref{lem:onebigger}, we simply need to prove its contrapositive.  To do this, we assume that there exists at least one entry in $\hat{\alpha}$ which is smaller than all entries in $\hat{\beta}$. We prove that, under this assumption, every term in the $H$-basis expansion of $\I_{\alpha/\beta}$ is zero.

Definition~\ref{def:skewdef} states that $\I_{\alpha/\beta}$ is given by the non-commutative determinant of the matrix $M_{\alpha/\beta}$ such that $(M_{\alpha/\beta})_{i,j}=H_{\alpha_i-i -(\beta-j)}$.  We also know that $H_p=0$  for all $p<0$.  If the matrix $M_{\alpha/\beta}$ contains an entire row of H functions with negative subscript, every term in the determinant must use an entry from this row and therefore $\I_{\alpha/\beta}=0$. 

Let $i$ be the subscript such that $\hat{\alpha}_i < \hat{\beta}_j$ for all $1 \le j \le \ell$.  Then $\alpha_i-i-(\beta_j-j)<0$ for all $1 \le j \le \ell$, so $H_{\alpha_j-j - (\beta_i-i)} =0$ for  all $1 \le j \le \ell$.  Therefore row $i$ in matrix $M_{\alpha/\beta}$ consists of all zero terms and the non-commutative determinant is $0$.
\end{proof}

Lemma~\ref{lem:onebigger} is the case $k=1$ of Theorem~\ref{thm:zeros}.  We now expand Lemma~\ref{lem:onebigger} to larger $k$ values. To do so, we need to consider the non-commutative determinant as a permutation.  To calculate each term of the determinant of an $n \times n$ matrix, the Laplace expansion is equivalent to selecting one entry from the first row, then an entry from the second row appearing in a different column than the entry selected from the first row, and so on.  This means that each column is selected exactly once and the entries selected for a specific term in the determinant form a permutation matrix when the selected entries are given a value of $1$ and all other entries are given the value $0$.

\begin{ex}
Let $\alpha=(4,1,6,5)$ and $\beta=(2,1,3,2)$. 
 The associated matrix is $M_{\alpha/\beta}=\begin{bmatrix} H_2 & H_4 & H_3 & H_5 \\ H_{-2} & H_{0} & H_{-1} & H_1 \\ H_{2} & H_{4} &H_{3} & H_5 \\ H_{0}&H_{2}&H_{1}&H_3\end{bmatrix}. \hspace*{0.5in}$

The term $-H_4H_1H_{2}H_1$ is one of the terms appearing in the Laplace expansion of the determinant for $\I_{\alpha/\beta}$.  This term is obtained by selecting the the second entry in the first row, the fourth entry in the second row, the first entry in the third row and the third entry in the fourth row.  Therefore it corresponds to the permutation $2413$.

On the other hand, $H_2H_1H_{5}H_0$ is not a term in the expansion although it does take a term from each row, because $H_1$ and $H_5$ reside in the same column.
\end{ex}

We now introduce another property of the matrix used to construct the skew immaculate functions. 

\begin{defn}
A matrix $M$ whose entries are of the form $H_n$ is said to have the \emph{negative-crossing property} if for every $2 \times 2$ submatrix $\begin{bmatrix} H_a & H_b \\ H_c & H_d \end{bmatrix}$, the following two statements are always true.  
\begin{enumerate}
\item If $a<0, b \ge 0,$ and $c \ge 0$, then $d \ge 0$.  
\item If $a \ge 0, b < 0,$ and $c < 0$, then $d < 0$.
\end{enumerate}
\end{defn}

We illustrate this property with the following two sub-matrices of $M_{\alpha/\beta}$.  Here, we use inequality signs to indicate the relationships between the subscripts and zero.

 $$\begin{bmatrix} H_{<} & H_{\ge} \\ H_{\ge} & H_{x}\end{bmatrix} \hspace*{0.5in}  \begin{bmatrix} H_{\ge} & H_{<} \\ H_{<} & H_{y}\end{bmatrix} \hspace*{0.5in}$$  If our matrix has the negative-crossing property, we must have the subscripts $x \ge 0$ and $y < 0$.  We now prove that every associated matrix $M_{\alpha/\beta}$ has the negative-crossing property.

\begin{theorem}{\label{thm:neg-cross}}
Let $\alpha$ and $\beta$ be compositions of length $\ell$.  Then the matrix $M_{\alpha/\beta}$ associated to $\I_{(\alpha/\beta)}$ has the negative-crossing property.  That is, assume $m, n, r,s  \in(1,2,3,...,l)$ with $m <n$ and $r<s$. 
 If $M_{r,m}=0, M_{r,n} \neq  0,$ and $M_{s,m}\neq0$, then $M_{s,n} \neq 0.$  Also if $M_{r,m}\neq0, M_{r,n} =  0$, and $M_{s,m}=0$, then $M_{s,n} = 0.$
\end{theorem}

\begin{proof}
Assume $M_{r,m}=0$.  This means $(\alpha_r-r)-(\beta_m-m)<0$.  If $ M_{r,n}\neq0$ for $n>m$, then $(\alpha_r-r)-(\beta_n-n)\geq0$. Rearrange these two equations to obtain $\beta_m-m>\alpha_r-r$ and $\alpha_r-r \ge \beta_n-n$. Combining these two expressions produces the inequality $\beta_m-m > \beta_n-n$. Therefore,  for all $s$ in $\{1,2,3,...,\ell \}$, $(\alpha_s-s)-(\beta_n-n)>(\alpha_s-s)-(\beta_m-m)$. If $M_{s,m}\neq0$, then $(\alpha_s-s)-(\beta_m-m)\geq0$. By the previous inequality, $(\alpha_s-s)-(\beta_n-n)>0$ and therefore $M_{s,m}\not= 0$.

Similarly, assume $M_{r,m} \not=0, M_{r,n} =  0$, and $M_{s,m}=0$.  $M_{r,m} \not=0$ implies that $(\alpha_r-r)-(\beta_m-m) \ge 0$.  Since $M_{r,n}=0$, we must have $(\alpha_r-r)-(\beta_n-n) < 0$.  Rearranging these two equations produces $\alpha_r-r \ge \beta_m-m$ and $\alpha_r-r < \beta_n-n$.  Combining these two expressions results in the inequality $\beta_m-m < \beta_n-n$.  Therefore $(\alpha_s-s)-(\beta_n-n) < (\alpha_s-s)-(\beta_m-m)$ for all $s$ such that $1 \le s \le \ell$.  Since $M_{s,m}=0$ and the subscript in $M_{s,n}$ must be smaller than the subscript in $M_{s,m}$, it must be the case that $M_{s,n}$ has a negative subscript, as desired.
\end{proof} 

\begin{cor}{\label{cor:zerocols}} 
In the associated matrix $M_{\alpha/\beta}$ for $I_{(\alpha/\beta)}$ (with $\alpha$ and $\beta$ compositions of length $\ell$), let $j,k \in\{1,2,...,\ell\}$.  Assume there exist $j$ rows $R_1, R_2, \hdots , R_j$ such that each row $R_i$ contains at least $k$ zeros.  Let $C_i=\{a_{i,1}, a_{i,2}, \hdots , a_{i,n_i} \}$ be the set of columns such that the entry in the column $a_{i,m}$ of row $R_i$ is equal to zero for all $m \in \{1, 2, \hdots , n_i \}$.  (Note that the number of elements $n_i$ in $C_i$ varies as $i$ varies, but is always greater than or equal to $k$.)  Then $\displaystyle{\cap_{i} C_i}$ has cardinality greater than or equal to $k$.
\end{cor}

\begin{proof}
This corollary follows immediately from Theorem $\ref{thm:neg-cross}$; we prove this by contradiction. First, it is clear that for any set of $j$ rows, if we can prove this result for the largest $k$ possible (meaning each row contains at least $k$ zero terms while some rows contain exactly $k$ zero terms), then the corollary immediately follows for any smaller value of $k$. 
    
Start with an arbitrary set $R_1, R_2, \hdots , R_j$ of rows satisfying the hypotheses and let $k$ be the largest $k$ such that each of these rows contains at least $k$ entries equal to zero. 
 At least one of these $j$ rows contains exactly $k$ zero terms.  Let $R_a$ be the highest row with exactly $k$ zero terms.

%Assume there exists a set of $j$ rows $R_1, R_2, \hdots , R_j$, each containing at least $k$ entries equalling zero, such that the sub-matrix of these $j$ rows contains less than $k$ columns whose entries are all equal to $0$.  For these $j$ rows, let the $k$ be the largest possible $k$ for the above reasoning, as a result, some row out the $j$ rows must contain exactly $k$ zero terms, we name this row $r_1$.

Assume $|\cap_i C_i| < k$ to get a contradiction.  Then some row $R_b$ must contain a nonzero entry in one of the $k$ columns in $C_a$; call this column $c_a$.  Since $R_b$ must contain at least $k$ zero terms, $R_b$ must contain a zero term in a column that does not contain a zero term in row $R_a$, call this column $c_b$.  Therefore the submatrix consisting of rows $R_a$ and $R_b$ and columns $c_a$ and $c_b$ must be represented by either $\begin{bmatrix} H_{\ge} & H_{<}\\ H_{<} &H_{\ge}\end{bmatrix}$ or $\begin{bmatrix} H_{<} & H_{\ge} \\ H_{\ge} & H_{<}\end{bmatrix}$.  Both of these configurations violate the negative-crossing property.  This results in a contradiction, since Theorem~\ref{thm:neg-cross} states that an associated matrix $M_{\alpha/\beta}$ cannot violate the negative-crossing property.
\end{proof}

The diagram in Figure~\ref{fig:cross} shows two rows, each containing three negative subscripts, but two of them at different column positions. Notice that columns 1 and 4 violate Theorem~\ref{thm:neg-cross}.

\begin{figure}
$$\begin{bmatrix} H_{\ge} & H_{<} & H_{<}& H_{<}& H_{\ge}\\ H_{<} & H_{<}& H_{<}& H_{\ge}& H_{\ge}\end{bmatrix}$$
\caption{Two rows with three zero terms but only two all-zero columns}{\label{fig:cross}}
\end{figure}

We have already considered the situation in which the associated matrix of $\I_{(\alpha/\beta)}$ contains an entire row of negative terms. Now we consider the case in which the matrix has a row $r_1$ with $n-1$ negative terms and one non-negative term in column $c_1$. One could produce a non-negative term by taking the entry in row $r_1$ and column $c_1$.  However, consider what happens when we have two rows $r_1$ and $r_2$ of $M_{\alpha / \beta}$ with exactly $n-1$ negative entries each.  By Corollary~\ref{cor:zerocols}, the non-negative entry on both rows must be in the same column $c_1$. We know that each term in the immaculate function must take one entry from $r_1$ as well as one entry from $r_2$. Therefore, if we use the non-negative entry from row $r_1$ and column $c_1$, then the entry from row $r_2$ must be from a column other than $c_1$ since two entries from the same column can't be selected for the same term. Therefore, the entry from row $r_2$ must have a negative subscript, which automatically makes the term zero since $H_a=0$ whenever $a <0$.
    
We now generalize this approach to prove Theorem~\ref{thm:zeros} by contraposition. Recall the statement of Theorem~\ref{thm:zeros}. \\

{\bf Theorem ~\ref{thm:zeros}}.  {\it Let $\alpha=(\alpha_1,...,\alpha_{\ell})$ and $\beta=(\beta_1,...,\beta_{\ell})$ be compositions.  If there is at least one non-zero term (before cancellation) in the $H$-basis expansion of $\I_{\alpha/\beta}$, then for every $k$ between $1$ and $\ell$ inclusive, the set of all parts of $\hat{\alpha}$ which are smaller than at least $\ell-k+1$ parts of $\hat{\beta}$ has cardinality less than or equal to $k-1$.}

\begin{proof}  
We will prove the contrapositive of Theorem ~\ref{thm:zeros}.  Let $\alpha=(\alpha_1,...,\alpha_{\ell})$ and $\beta=(\beta_1,...,\beta_{\ell})$ be compositions and let $S_k = \{\hat{\alpha_i}|\hat{\alpha_i}$ is smaller than at least $l-k+1$ entries in $\hat{\beta}\}$. Assume there exists a $k$ such that $1 \le k \le \ell$ and $|S_k|\geq k$.  We will show that in this case, there are no non-zero terms in the $H$-basis expansion of $\I_{\alpha/\beta}$ at all.

The assumption that $|S_k| \ge k$ means that each of the $k$ selected rows includes at least $\ell-k+1$ negative subscripts.  
Corollary~\ref{cor:zerocols} implies that there are at least $\ell-k+1$ columns which contain only negative subscripts in these $k$ rows.  But this means there are at most $k-1$ column locations within these $k$ rows containing any non-negative subscripts.  By the Pigeonhole Principle, any selection of $k$ distinct column entries from among these $k$ rows must include a negative subscript, meaning the overall term is zero.  Therefore there cannot be any non-zero terms in the $H$-expansion.
\end{proof}

We now have proven Theorem~\ref{thm:zeros}, which provides a necessary condition for the existence of a non-zero term in the $H$-basis expansion of skew immaculate functions.  It is natural to ask whether that condition is sufficient to guarantee a non-zero term in the $H$-basis expansion. It turns out this condition is indeed sufficient, at least before cancellation. In the next section, we describe some background needed to prove this.

\section{Hall's Matching Theorem}{\label{sec:Hall}}

We now describe a well-known result in graph theory which we will use in our subsequent proofs.  To keep this article self-contained, we include several definitions before describing Hall's Matching Theorem.  A \emph{graph} is a collection of vertices $V=\{v_1, v_2, \hdots v_n\}$ and a set of edges $E$ which are pairs $\{v_i,v_j\}$ (with $i \not= j$) of vertices.  Note that we only consider simple graphs on finitely many vertices.

\begin{defn}
A \emph{bipartite graph} $G=(V,E)$ is a graph on two disjoint sets of vertices $V=V_1 \cup V_2$ such that every edge in $E$ includes one vertex from $V_1$ and one vertex from $V_2$.
\end{defn}

In other words, there are no edges between vertices from the same set. Since we are only considering simple graphs, there is at most one edge connecting any two vertices.  Bipartite graphs are sometimes denoted by their vertex sets $G(V_1,V_2)$.

\begin{defn}
In a bipartite graph, a \emph{complete matching} on one of the vertex sets $V_1$ is a set of $|V_1|$ edges $E'$ such that no vertex appears in more than one edge in $E'$.
\end{defn}

Informally, a complete matching in $G=(V,E)$ is a subgraph of $G$ consisting of a set of $2|V_1|$ vertices and $|V_1|$ edges such that each vertex has degree one.

One important result in graph theory is Hall's Matching Theorem, which provides necessary and sufficient conditions for a bipartite graph to admit a complete matching as a subgraph.

\begin{theorem}[Hall's Matching Theorem~\cite{Wil96}]  Let $G = G(V_1, V_2)$ be a bipartite graph.  For each subset $S$
of $V_1$, let $\rho(S)$ be the set of vertices of $V_2$ that are adjacent to at least one vertex of
$S$. Then a complete matching from $V_1$ to $V_2$ exists if and only if $|S|\leq|\rho(S)|$ for
each subset $S$ of $V_1$.
\end{theorem}

Consider the graph $G(V_1,V_2)$ with vertex sets $V_1=\{A,B,C,D,E\}$ and $V_2=\{1,2,3,4,5,6\},$ and edge set $E$ depicted below.  In this graph, for every subset $S$ of $k$ letters, the set $\rho(S)$ has cardinality greater than or equal to $|S|$.  For example, the set $\{A, B\}$ corresponds to $\rho(\{A,B\})=\{1,2,4 \}$ and $| \{1,2,4\}| =3 > | \{A,B\}|$.  Therefore this graph admits a complete matching as a subgraph, such as the matching described by the edge set $\{\{A,2\}, \{B,4\},\{C,5\},\{D,3\},\{E,6\}\}$. 

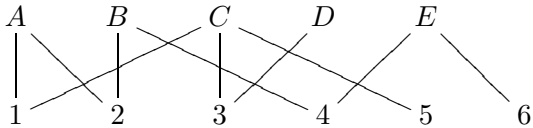
\begin{center}
\xymatrix{
A \ar@{-}[d] \ar@{-}[dr] & B \ar@{-}[d] \ar@{-}[drr] & C \ar@{-}[dll] \ar@{-}[d] \ar@{-}[drr] & D \ar@{-}[dl] & E \ar@{-}[dl] \ar@{-}[dr]  \\
1 & 2 & 3 & 4 & 5 & 6
}
\end{center}

However, the following graph does not satisfy the hypotheses of Hall's Matching Theorem.  Consider the subset $S=\{A,B,D \}$.  The set $\rho(\{A,B,D\}=\{2,4\}$ has cardinality $2$, which is less than the cardinality of $\{A,B,D\}$.  Therefore this graph does not admit a complete matching since there are not enough options for $A,B,$ and $D$.

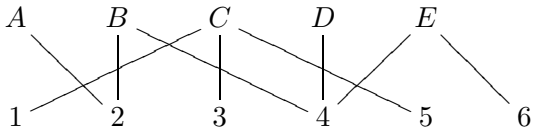
\begin{center}
\xymatrix{
A \ar@{-}[dr] & B \ar@{-}[d] \ar@{-}[drr] & C \ar@{-}[dll] \ar@{-}[d] \ar@{-}[drr] & D \ar@{-}[d] & E \ar@{-}[dl] \ar@{-}[dr]  \\
1 & 2 & 3 & 4 & 5 & 6
}
\end{center}

\section{A sufficient condition for non-zero terms in a skew immaculate}{\label{sec:sufficient}}

We now state the converse of Theorem~\ref{thm:zeros}, which we will prove by contraposition.  Note that although these two theorems taken together provide a necessary and sufficient condition for a non-zero term to appear in the $H$-basis expansion of $\I_{\alpha/\beta}$, any such nonzero term might still be subject to cancellation.  Therefore this theorem does not guarantee that $\I_{\alpha/\beta}$ is nonzero, but it does allow us to rule out a large class of composition pairs whose skew immaculate functions are all zero.

\begin{theorem}{\label{thm:reverse}}
Let $\alpha$ and $\beta$ be compositions. If for every $k$ between $1$ and $\ell$ inclusive, the set of all parts of $\hat{\alpha}$ which are smaller than at least $\ell-k+1$ parts of $\hat{\beta}$ has cardinality less than or equal to $k-1$, then there exists at least one nonzero term in the H-basis expansion before cancellation. 
\end{theorem}

\begin{proof}
The condition in the theorem is equivalent to requiring that for each set of rows of size $k \in \{1,2, \hdots , \ell \},$ there exists at least one row that contains at least $k$ nonnegative entries.  We shall prove that at least one term of the expansion of $\I_{\alpha/\beta}$ is nonzero before cancellation.

To find such a nonzero term, first note that this term will need to include an entry with a non-negative subscript from each row.  We use Hall's Matching Theorem to find the required term. To do so, we first treat each row of $M_{\alpha/\beta}$ as a vertex, so that this set $R$ is a set of size $l$.  We then treat the columns as another vertex set $C$, also of size $l$.  A pair $\{r,c\}$ (with $r \in R$ and $c \in C$) is an element of the edge set if and only if row $r$ contains a nonzero entry in column $c$. See Example \ref{hallex} for an associated matrix and its interpretation as a bipartite graph.

Finding a nonzero term in the $H$-basis expansion of a skew immaculate therefore corresponds to finding a complete matching between the row set $R$ and the column set $C$. This is because if a complete matching between the two sets exists, the edges in the matching tell us which row and column to select to produce a nonzero term. By the hypotheses of this theorem, given $k \in \{1,2, \hdots , \ell \}$, we know that for each subset of rows of size $k$ of the associated matrix, the corresponding collection of adjacent column vertices must be at least size $k$. 
 This is precisely the hypothesis needed for Hall's Theorem.  Hall's Theorem therefore implies that this condition is sufficient for a complete matching to exist for the graph $G(R,C)$. 
\end{proof}

\begin{ex}{\label{hallex}}
For example, let $\alpha=(10,7,9)$ and $\beta=(9,8,5)$ so that the associated matrix is 

$$M_{\alpha/\beta}=\begin{bmatrix} 
H_{1} & H_{3} & H_{7}  \\ 
H_{\shortminus 3} & H_{\shortminus 1} & H_{3}  \\ 
H_{\shortminus 2} & H_{0} & H_{4} \\ 
\end{bmatrix}.$$

The bipartite graph corresponding to this situation is shown below.

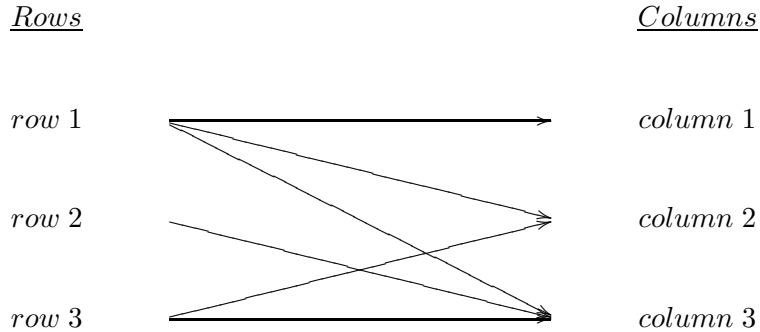
\xymatrix{
\underline{Rows} & & & & & & & \underline{Columns} \\ 
row \; 1 & \ar@[red][rrrrr] \ar@[black][rrrrrd] \ar@[black][rrrrrdd]& & & & & & column \; 1 \\
row \; 2 & \ar@[red][rrrrrd] & & & & & & column \; 2 \\
row \; 3 & \ar@[red][rrrrru]\ar@[black][rrrrr]& & & & & & column \; 3 }

\vspace*{.2in}

Here, the vertex corresponding to row $1$ is adjacent to all three column vertices since all entries in row $1$ have nonnegative subscripts.  Row $2$, however, only has a nonnegative subscript in column $3$ and is therefore only adjacent to the column $3$ vertex.  Similarly, the row $3$ vertex is only adjacent to the vertices corresponding to columns $2$ and $3$ since the subscript in column $1$ of row $3$ is negative.  The complete matching consisting of edges $\{ r_1c_1, r_2c_3, r_3c_2 \}$ is shown in the matrix below with the selected entries in bold.

$$\begin{bmatrix} 
\bf{H_{1}} & H_{3} & H_{7}  \\ 
H_{\shortminus 3} & H_{\shortminus 1} & \bf{H_{3}}  \\ 
H_{\shortminus 2} & \bf{H_{0}} & H_{4} \\ 
\end{bmatrix} $$

Such a matching is guaranteed to exist if the hypotheses of Theorem~\ref{thm:reverse} are met since these conditions are equivalent to those of Hall's Matching Theorem.
\end{ex}

\section{A class of composition pairs for which the corresponding immaculate function is non-zero}{\label{sec:cancel}}

Theorems \ref{thm:zeros} and \ref{thm:reverse} provide a necessary and sufficient condition under which the skew immaculate function $\I_{\alpha/\beta}$ has at least one non-zero term in its homogeneous basis expansion before performing any cancellation.  However, due to potential cancellations, this is not enough to ensure the entire expansion is non-zero. We can see this in the following example.

\begin{ex}{\label{ex:cancels}}
Let $\alpha=(9,5,5),\beta=(2,5,6),$ then $\hat{\alpha}=(8,3,2),\hat{\beta}=(1,3,3)$. So for $M_{\alpha/\beta}$, we have the following matrix:

$$\begin{bmatrix} H_7 & H_5 & H_5 \\ H_2 & H_{0} & H_0\\ H_{1} & H_{-1} &H_{-1} \end{bmatrix} \hspace*{0.5in}$$
\end{ex}
In this matrix, since the last row contains two negative entries, we have to avoid using them to produce a non-negative term. Therefore,we get $\I_{\alpha/\beta}=H_5H_0H_1-H_5H_0H_1$ by fixing $H_1$ for the last row. Clearly these two terms cancel each other out since they both equal $H_{51}$ but have opposite signs.

Before considering cancellations, we first attempt to reduce the problem by restricting the indexing shapes. Any skew shape can be obtained from a shape skewed by a partition via a series of ``straightening operations"~\cite{AllMas23}.  This procedure allows us to limit our focus to skewing only by partitions instead of by arbitrary compositions.  Allen and Mason's formula tells us that if $\I_{\alpha/\lambda}=0$ (for a partition $\lambda$), then any other skew shape $\alpha/\beta$ obtained from $\alpha/\lambda$ via this transformation will also satisfy $\I_{\alpha/\beta} =0$.  Therefore, it is enough to focus only on skewing by partitions.  The matrix $M_{\alpha/\lambda}$ (with $\alpha$ a composition and $\lambda$ a partition) has several useful properties.  These are summarized in the following two lemmas. 

\begin{lem}{\label{lem:IDcols}}
    In the matrix $M_{\alpha/\lambda}$ associated to composition $\alpha$ and partition $\lambda$, no two columns can be identical.  In fact, the subscripts of the column entries strictly increase along rows from left to right.
\end{lem}
\begin{proof}
Since $\lambda$ is a partition, its entries must weakly decrease from left to right, meaning $\lambda_i\geq\lambda_j $ for all $i,j$ such that $ 1\leq i<j\leq l$. Therefore, when $i<j$, we have $\lambda_i-i>\lambda_j-j$.  Since $\hat{\lambda}_i=\lambda_i-i $ and $\hat{\lambda}_j=\lambda_j-j$, we have $\hat{\lambda}_i>\hat{\lambda}_j$ for all $i < j$.  This means no two columns in the matrix can be the same and in fact since $(M_{\alpha/\lambda})_{r,i}=H_{\hat{\alpha_r}-\hat{\lambda}_i}$ and $(M_{\alpha/\lambda})_{r,j}=H_{\hat{\alpha_r}-\hat{\lambda}_j}$, the subscripts increase along rows from left to right.
\end{proof}

The following two corollaries are immediate consequences of Lemma~\ref{lem:IDcols}.

\begin{cor}{\label{cor:split}}
 In the matrix $M_{\alpha/\lambda}$ associated to composition $\alpha$ and partition $\lambda$, for any row of length $l$ with exactly $k$ non-negative subscripts, the first $l-k$ entries in the row must be negative while the last $k$ entries must have non-negative subscripts.
\end{cor}

\begin{proof}
    Since the subscripts in the row strictly increase from left to right by Lemma \ref{lem:IDcols}, all subscripts to the right of any non-negative subscript must be non-negative while all subscripts to the left of any negative subscript must be negative. Therefore, the first $(l-k)$th subscripts must be negative to ensure there are at least $l-k$ negative subscripts while the $(l-k+1)$th subscript must be non-negative to ensure there are at least $k$ non-negative subscripts. Combining both statements, the first $l-k$ subscripts in this row must be negative and the remaining $l-k+1$ subscripts must be non-negative.
\end{proof}

\begin{cor}{\label{cor:follow}}
    For a composition $\alpha$ and partition $\lambda$, if $H_0$ appears in column $t$ of row r of $M_{\alpha /\lambda}$, then for $H_b$ in column $j$ of row $r$, we have $b<0$ if $j<t$ and $b>0$ if $ j>t$.
\end{cor}

After applying the restriction to make sure there is at least one non-zero term and skewing only by partitions, there are still some pairs $\alpha$ and $\lambda$ that result in $\I_{\alpha / \lambda} =0$. We observe that the associated matrix of such pairs seems to always contain two identical rows each containing the entry $H_0$. The following compositions and their associated matrix form an example. Let $\alpha=(2,2,5,5)$ and $\lambda=(3,3,3,3)$.  Then $\hat{\alpha}=(1,0,2,1)$ and $\hat{\lambda}=(2,1,0,-1)$, so that $$M_{\alpha /\lambda}=\begin{bmatrix} H_{-1} & H_0 & H_1 & H_2 \\ H_{-2} & H_{-1} & H_0 & H_1 \\ H_{0} & H_{1} &H_{2} & H_3 \\ H_{-1}&H_{0}&H_{1}&H_2\end{bmatrix}.$$

Notice rows 1 and 4 of the matrix are identical and both have second entry $H_0$. Therefore, the term $H_0H_0H_0H_2$ will be cancelled out by  $-H_2H_0H_0H_0$ and the term $H_1H_1H_0H_0$ will be cancelled by $-H_0H_1H_0H_1$, since $H_0=1$ and the signs are different in each pair.  This observation motivates the following theorem.

\begin{theorem}{\label{thm:nocancel}}
Let $\alpha$ be a composition and $\lambda$ be a partition (both of length $\ell$), and let $M_{\alpha / \lambda}$ be the associated matrix. If both of the following conditions are satisfied, then $\I_{\alpha / \lambda} \neq 0$.

\begin{enumerate}
\item In any set $S$ of $k$ rows of $M_{\alpha /\lambda}$, at least one row in $S$ contains at least $k$ non-negative subscripts.
\item If rows $i$ and $j$ are identical, then no entry in these rows is $H_0$.
\end{enumerate}

\end{theorem}

\begin{proof}
We first prove that if $M_{\alpha / \lambda}$ satisfies the conditions of Theorem~\ref{thm:nocancel} then the $H$-basis expansion of $\I_{\alpha/ \lambda}$ includes a non-zero term containing every occurrence of $H_0$ in $M_{\alpha /\lambda}$.  After that we will show that any such term cannot be canceled out.

The following algorithm constructs a term in the $H$-basis expansion of $\I_{\alpha/\lambda}$ which includes every occurrence of $H_0$.  

\begin{enumerate}
\item Select row $r_1$ with $\ell$ non-negative subscripts.  (Condition (1) implies that such a row exists.)
\begin{enumerate}
\item If there are multiple such rows, select the row containing $H_0$ if such a row exists.  (Corollary $\ref{cor:follow}$ and Condition (2) of Theorem~\ref{thm:nocancel} imply that there is at most one row with $\ell$ non-negative subscripts containing $H_0$. In addition, $H_0$ must be the first entry in that row.)
\item If there are multiple rows with $\ell$ non-negative subscripts and none contain a copy of $H_0$, select the highest row containing $\ell$ non-negative subscripts.  (In fact, selecting any of these rows will suffice.) 
\end{enumerate}
\item Select the term in the leftmost column of row $r_1$.
\item Delete row $r_1$ and column $1$ to obtain the matrix $M^{(1)}_{\alpha / \lambda}$.  (We will prove that $M^{(1)}_{\alpha / \lambda}$ satisfies the conditions of Theorem~\ref{thm:nocancel} since removing a row and column does not change these conditions.)
\item Repeat steps (1) through (3) for the matrix $M^{(1)}_{\alpha / \lambda}$, replacing $\ell$ with $\ell-1$ and $r_1$ with $r_2$ and finishing with the matrix $M^{(2)}_{\alpha / \lambda}$.
\item Continue this process until $\ell$ entries have been selected.
\end{enumerate}

This process selects every occurrence of $H_0$ since any row containing $H_0$ in column $c$ will be among those rows with $\ell-c+1$ non-negative entries and by Step (1b) of this algorithm will be the row chosen.  Therefore $H_0$ will be the term selected from this row.

Now we show that the deletion in Step 3 of this process preserves the conditions in the hypothesis of Theorem~\ref{thm:nocancel}.  First, it is clear that the second hypothesis of Theorem~\ref{thm:nocancel} remains true since deletion does not alter the rows that remain.  We use induction on the number of deletions to prove that the first hypothesis remains true.  The base case is the situation in which no deletions were made, so the hypotheses are trivially satisfied.

Next consider a set $S$ of $k$ rows of $M^{(i)}_{\alpha / \lambda}$, where $k \le \ell-i$ and $1 \le i \le \ell-1$.  (If $i=\ell$ then all rows and columns would be deleted and there would be nothing to check.)  By the inductive hypothesis, in the matrix $M^{(i-1)}_{\alpha / \lambda}$, at least one of these rows contains at least $k$ non-negatives by the first hypothesis of Theorem~\ref{thm:nocancel}.  By Corollary $\ref{cor:split}$, the last $k$ subscripts in that row must be non-negative.  Therefore, deletion of the first column would leave $k$ non-negative subscripts and the hypothesis remains true.

Next we prove that any non-zero term containing all occurrences of $H_0$ cannot be canceled out.  To see this, let $H_{\gamma}$ be the term containing every occurrence of $H_0$. 
 (Note that if $M_{\alpha / \lambda}$ contains no occurrences of $H_0$, every term whose subscripts are all positive satisfies this condition.)  Since $H_0=1$, we can write this term as $x:=(-1)^{sign(\sigma)}H_{a_1} H_{a_2} \cdots H_{a_{k}}$ such that $k \le \ell$ and $a_i >0$ for all $i$ and $\sigma$ is the permutation corresponding to this selection.  (If $a_i=0$ for all $1 \le i \le \ell$, then $x= \pm 1$.)

We now assume there exists another term $y$ that cancels out $x$ to get a contradiction.  Then $y$ must be in the form $y=(-1)^{sign(\sigma)+1}H_{a_1}H_{a_2} \cdots H_{a_k}$, otherwise it cannot cancel out $x$. Since the number of non-zero terms must be the same in $x$ and $y$, the term $y$ must also include all occurrences of $H_0$. Since $H$ functions are not commutative, $H_{a_1}$ must be chosen at the same row in $y$ as in $x$. By Lemma \ref{lem:IDcols}, there are no repeated entries within a row.  Therefore the $H_{a_i}$ appearing in $y$ must be chosen from the same row and column as in $x$ for each $1 \le i \le k$.  But then the sign of $y$ must equal the sign of $x$ and in fact they are the same term.  Since each term can only be selected once, this situation cannot occur and therefore $x$ cannot be canceled out.
\end{proof}

In this article, we have constructed a class of pairs of compositions for which the skew immaculate function is nonzero.  Although not every nonzero skew immaculate function is included in this class, it does provide a valuable setting for computations involving multiplication.  An important direction for future research is to identify all the pairs of compositions for which the skew immaculate is nonzero.

\bibliographystyle{amsplain}

\bibliography{MasonXiebib}

\providecommand{\bysame}{\leavevmode\hbox to3em{\hrulefill}\thinspace}
\providecommand{\MR}{\relax\ifhmode\unskip\space\fi MR }
\providecommand{\MRhref}[2]{%
  \href{http://www.ams.org/mathscinet-getitem?mr=#1}{#2}
}
\providecommand{\href}[2]{#2}
\begin{thebibliography}{1}

\bibitem{AllMas23}
Edward~E Allen and Sarah~K Mason, \emph{A combinatorial interpretation of the
  noncommutative inverse kostka matrix}, 2023.

\bibitem{BBSSZ14}
Chris Berg, Nantel Bergeron, Franco Saliola, Luis Serrano, and Mike Zabrocki,
  \emph{A lift of the {S}chur and {H}all-{L}ittlewood bases to non-commutative
  symmetric functions}, Canad. J. Math. \textbf{66} (2014), no.~3, 525--565.
  \MR{3194160}

\bibitem{Mac95}
I.~G. Macdonald, \emph{Symmetric functions and hall polynomials}, 2nd ed ed.,
  Oxford mathematical monographs, Clarendon Press; Oxford University Press,
  1995.

\bibitem{Sag01}
Bruce~E. Sagan, \emph{The symmetric group}, second ed., Graduate Texts in
  Mathematics, vol. 203, Springer-Verlag, New York, 2001, Representations,
  combinatorial algorithms, and symmetric functions. \MR{1824028}

\bibitem{Sta99}
Richard~P. Stanley, \emph{Enumerative combinatorics. {V}ol. 2}, Cambridge
  Studies in Advanced Mathematics, vol.~62, Cambridge University Press,
  Cambridge, 1999, With a foreword by Gian-Carlo Rota and appendix 1 by Sergey
  Fomin. \MR{1676282}

\bibitem{Wil96}
Robin~J. Wilson, \emph{Introduction to graph theory}, fourth ed., Longman,
  Harlow, 1996. \MR{2590569}

\end{thebibliography}
\label{sec:biblio}

\end{document}